\documentclass[11pt]{article}
\usepackage[utf8]{inputenc}
\usepackage{graphicx}
\usepackage{natbib}
\usepackage[hidelinks]{hyperref}
\hypersetup{colorlinks=false, urlcolor=black}
\usepackage{mathtools}
\usepackage{float}
\usepackage{enumerate}
\usepackage{mathtools}
\usepackage{scalerel,stackengine}
\usepackage{amssymb,amsmath,amsthm}
\usepackage{relsize}
\usepackage{mathrsfs}
\usepackage{algorithm}
\usepackage{algpseudocode}
\usepackage{caption}
\usepackage{subcaption}

\newtheorem{theorem}{Theorem}
\newtheorem{lemma}{Lemma}
\newtheorem{proposition}{Proposition}

\newcommand\gfrac[1]{\Gamma \left (\frac{#1}{2} \right )}

\newcommand{\prob}[0]{\textbf{Pr}}

\begin{document}

\title{\bf Concentration inequalities for the sample correlation coefficient}
  \author{
    Daniel Salnikov\thanks{Department of Mathematics, Huxley Building, Imperial College, 180 Queen's Gate, South Kensington, London, SW7 2AZ, UK.} 
    \thanks{Great Ormond Street Institute of Child Health, 30 Guilford Street, London WC1N 1EH, PEP DRIVE}\footnotemark[2]
    \\
    Imperial College London   
    }

\maketitle

\begin{abstract}
    The sample correlation coefficient $R$ plays an important role in many statistical analyses. We study the moments of $R$ under the bivariate Gaussian model assumption, provide a novel approximation for its finite sample mean and connect it with known results for the variance. We exploit these approximations to present non-asymptotic concentration inequalities for $R$. Finally, we illustrate our results in a simulation experiment that further validates the approximations presented in this work.
\end{abstract}    
    \noindent%
    \newline moments, sub-Gaussian bound.

\section{Introduction}
\label{sec: intro}
The sample correlation coefficient $R$ plays an important role in many statistical analyses. It is defined for samples with size $n \geq 3$, under the bivariate normal model assumption Fisher obtained an expression for the density function and proposed a $Z$-transform that is approximately normal; see \cite{fisher_corr, r_moments}. However, the associated normal approximation requires large sample sizes to be valid, and the asymptotic variance of the $Z$-transform does not depend on $\rho$; see \cite{hotelling, wimberton}. To alleviate these shortcomings, \cite{hotelling} and \cite{r_moments} derived closed form expressions for the moments of $R$ that depend on the population correlation coefficient $\rho$ and sample size $n$. We build on this work by proposing simplified bounds for the mean and variance that depend on $\rho$ and $n$. Further, we exploit these bounds by deriving non-asymptotic concentration inequalities for $R$.

\section{Mean and variance approximations}
\label{sec: mean and variance approximation}
We are interested in approximating the mean of the Pearson sample correlation coefficient computed from a sample of random variables $(X_i, Y_i) \in \mathbb{R}^2$, for $i \in \{ 1, \dots, n\}$. The sample correlation coefficient $R$ is given by
\begin{equation*}
   R = \frac{\sum_{i = 1}^{n} \big ( X_i - \overline{X} \big ) \big ( Y_i - \overline{Y} \big )}{
   \big \{ \sum_{i = 1}^{n} \big ( X_i - \overline{X} \big ) \big \}^{ 1/2 }
   \big \{ \sum_{i = 1}^{n} \big ( Y_i - \overline{Y} \big ) \big \}^{ 1/2 }} .
\end{equation*}
If we assume that $(X_i, Y_i)$ have a bivariate normal distribution, i.e.,  
    $(X_i, Y_i) \sim \mathcal{N}_2 ( \boldsymbol{\mu}, \mathbf{\Sigma} ),$ 
where $ \boldsymbol{\mu} = (\mu_x, \mu_y) \in \mathbb{R}^2$ and $\mathbf{\Sigma} \in \mathbb{R}^{2 \times 2}$ is the covariance matrix such that $\rho = \sigma_{xy} (\sigma_{x} \sigma_{y} )^{-1} \in [-1, 1]$ is the population correlation coefficient, $\sigma_x$, $\sigma_y > 0$ are the population standard deviations and $\sigma_{xy} \in \mathbb{R}$ is the covariance, then the probability density function of $R$ is given by
\begin{equation*}
    f_{R} (r) = \frac{2^{n - 3} (1 - \rho^2)^{\frac{n - 1}{2}} (1 - r^2)^{\frac{n - 4}{2}} }{
        \pi \Gamma(n - 2) 
    } \sum_{k = 0}^{+\infty} \bigg [\gfrac{n - 1 + k} \bigg ]^2 \frac{\big (2 r \rho \big )^k}{k!}; 
\end{equation*} 
see \cite{fisher_corr, hotelling, r_moments}. We will use the following result, which combines results of \cite{hotelling} and \cite{r_moments} with a new approximation.
\begin{theorem} \label{th: moments}
Let $R$ be the sample correlation coefficient computed from a sample of size $n \geq 3$ coming from a bivariate normal distribution with population correlation coefficient $\rho$. Then, we have that
\begin{equation} \label{eq:m_moment_sum}
    E(R^m) = \frac{2^{n - 3} (1 - \rho^2)^{\frac{n - 1}{2}}}{
        \pi \Gamma(n - 2) 
    } \sum_{k = 0}^{+\infty} \bigg [\gfrac{n - 1 + k} \bigg ]^2 \frac{\big (2 \rho \big )^k}{k!} g_m(k),
\end{equation}
\begin{equation} \label{eq: mean approx}
    E(R) =  (1 - n^{-1})^{1/2} \rho + O (n^{-1}),
\end{equation}
\begin{equation}\label{eq:var_approx}
     \mathrm{var} (R) = \frac{(1 - \rho^2)^2}{n - 1} + O(n^{-2}),
\end{equation}
\begin{equation} \label{eq: second_bound}
    E(R^2) = \rho^2 + \frac{(1 - \rho^2)^2}{n - 1} + O(n^{-2}),
\end{equation}
where $g_m (k) = \left \{ \gfrac{m + k + 1} \gfrac{n - 2} \right \} \left \{
        \gfrac{n + m + k - 1} \right \}^{-1} \mathbb{I} \{m + k = 2 q \},$ $q$ is a non-negative integer and $\mathbb{I}$ is the indicator function.
\end{theorem}
\begin{proof}
See \cite{r_moments} for a proof of \eqref{eq:m_moment_sum} or the supplementary material for an alternative one. A proof of \eqref{eq: mean approx} follows from manipulating \eqref{eq:m_moment_sum}, for $\rho > 0$, into 
\begin{align} \label{eq:one_ineq}
    E(R) &\geq (1 - n^{-1})^{1/2} \rho \sum_{q = 0}^{+\infty} \frac{\Gamma(q + \frac{n - 1}{2} ) }{ \gfrac{n - 1}  \Gamma(q + 1) } (\rho^2)^q (1 - \rho^2)^{\frac{n - 1}{2}} = (1 - n^{-1})^{1/2} \rho, \nonumber \\
    E(R) &\leq \rho \sum_{q = 0}^{+\infty} \frac{\Gamma(q + \frac{n - 1}{2} ) }{ \gfrac{n - 1}  \Gamma(q + 1) } (\rho^2)^q (1 - \rho^2)^{\frac{n - 1}{2}} = \rho.
\end{align}
If $\rho < 0$, then the inequalities in (\ref{eq:one_ineq}) turn. Therefore, for $|\rho| < 1$ we conclude that $ (1 - n^{-1})^{1/2} |\rho| \leq |E(R)| \leq |\rho|$. Moreover, using \eqref{eq:one_ineq}, we find the error bound $1 - (1 - n^{-1} )^{-1/2} \leq n^{-1}$.
\par
See \cite{hotelling} for a proof of \eqref{eq:var_approx}. Using \eqref{eq: mean approx} and \eqref{eq:var_approx}, only taking the first two terms in \eqref{eq:m_moment_sum} when $m = 2$, and noting that $E(R^2) - E^2 (R) \geq 0$ for all $n$ we have that 
$$\rho^2 (1 - n^{-1}) + \frac{(1 - \rho^2 )^2}{n - 1} < E(R^2) \leq \rho^2 + \frac{(1 - \rho^2 )^2}{n - 1} + O(n^{-2}).$$
\end{proof}
Interestingly, the sample correlation coefficient $R$ will, on average, be closer to zero than $\rho$, nonetheless, as $n \to \infty$ the expected value gets closer and closer to $\rho$. Note that the variance is largest when $\rho = 0$ and that if $X$ and $Y$ are perfectly correlated (i.e., $\rho =\pm 1$), then the variance is equal to zero given that $R$ is a degenerate random variable in that case. This results in more skewed distributions as $|\rho|$ approaches one. The following bounds satisfy these properties. By Theorem \ref{th: moments} we have that
\begin{align} \label{eq: var upper bounds}
    \mathrm{var}(R) &\lesssim \frac{(1 - \rho^2 )^2}{n - 1} + \frac{(1 - \rho^2 )}{n - 1}, \\
    \mathrm{var}(R) &\lesssim \frac{2 (1 - \rho^2 )^2}{n - 1}, \nonumber
\end{align}
and by \eqref{eq:var_approx} for $n$ sufficiently large we have that $\mathrm{var}(R) \approx (1 - \rho^2 )^2 (n - 1)^{-1}$. Notice that the bound given by \eqref{eq: var upper bounds} is more conservative than the other one, hence, it is preferable for smaller samples.

\section{Concentration inequalities}
\label{sec: concentration inequalities}
We use the approximations given by \eqref{eq: mean approx} and \eqref{eq:var_approx} for studying concentration inequalities for $R$. Combining Markov's inequality, \eqref{eq: mean approx} and \eqref{eq: var upper bounds} gives 
\begin{equation} \label{eq: R consistency}
    R = \rho + O_{\prob} (n^{-1/2}).
\end{equation}
Tighter bounds can be achieved, the result follows below.
\begin{proposition} \label{prop: concentration inequality}
    Let $R$ be the sample correlation coefficient computed from a sample of size $n \geq 3$ coming from a bivariate normal distribution with population correlation coefficient $\rho \in (-1, 1)$. Then for any $t > 0$ we have that
    \begin{equation*}
        \mathrm{\mathbf{Pr}} ( |R - \rho| > t) \leq 2 \exp \big [- n t^2 \{2 (1 + 2 n t) \}^{-1} \big ],
    \end{equation*}
    moreover, for $n$ sufficiently large, we have the approximations
    \begin{equation} \label{eq: concervative bound}
         \mathrm{\mathbf{Pr}} ( |R - \rho| > t) \lesssim 2 \exp \left \{- \frac{n t^2}{8 (1 - \rho^2)^2 } \right \},
    \end{equation}
     \begin{equation} \label{eq: agressive bound}
        \mathrm{\mathbf{Pr}} ( |R - \rho| > t) \lesssim 2 \exp \left \{- \frac{n t^2}{4 (1 - \rho^2)^2 } \right \}.
    \end{equation}
\end{proposition}
\begin{proof}
    See the supplementary material for a complete proof, we present a summary with the main ideas. Since $R$ is bounded almost surely, all of its moments exist, moreover, $E \{ (R - \rho)^{2m} \} = E \{ (R - \rho)^{2m - 2} (R - \rho)^{2}\} \leq 2^{m -2} \nu^2$, where $m \in \{1, 2, 3, \dots\}$, and $\nu^2 := \mathrm{var} (R) \leq (n - 1)^{-1}$, where equality holds if $\rho = 0$. Hence, $R$ satisfies Bernstein's condition; see \cite{wainwright_2019}. Consequently, by \eqref{eq:var_approx}, for $t > 0$ we have that
    \begin{equation*}
         \prob ( |R - E(R) | > t ) \leq 2 \exp \big [-t^2 \{2 (\nu^2 + 2 t) \}^{-1} \big ],
    \end{equation*}
   scaling properly by $(1 - n^{-1})^{1/2}$ finishes the first part of the proof. We proceed to the second part.
   \par
    Notice that by looking at \eqref{eq:m_moment_sum} and expressing the gamma function as a factorial we have that 
    $g_{2m} (2 q) \leq \prod_{t = 1}^{m - 1} \left ( \frac{2t + 2q + 1}{2t + 2q + n - 1} \right ) g_2 (2q),$ which when substituted into \eqref{eq:m_moment_sum} results in the inequality
    \begin{equation*}
        \left ( 1 - \frac{n - 2}{n + 1} \right )^{m - 1} E(R^2) \leq E(R^{2m}) \lesssim \left ( 1 - \frac{n - 2}{2m + n - 1} \right )^{m - 1} E(R^2),
    \end{equation*}
    thus, we can bound from below and above even moments with the second moment, and approximate them. By combining the above and \eqref{eq: R consistency}, we find a sequence $(a_n) \in \mathbb{R}$ such that $a_n \geq 1$ and $a_n$ approaches one as $n \to +\infty$ (i.e., $a_n \to 1$), so that the three terms in: $\{E(R^2)\}^m \leq E(R^{2m}) \leq a_n \{E(R^2)\}^m$ for $m \in \{1, \dots \}$ converge to $\rho^{2m}$. Combining this inequality with \eqref{eq: mean approx} and \eqref{eq: second_bound}, and recalling that $|\rho| < 1$ we show that
    \begin{equation*}
        E \{ (R - \rho)^{2m} \} \lesssim \frac{(2m)!}{2^{m} m!} \left [ 2^{1/2m} 
        \left \{ E(R^2) - \rho^2 \right \}^{1/2} \right]^{2m} \leq  \frac{(2m)!}{2^{m} m!} (\sqrt{2} \nu)^{2m}.
    \end{equation*}
    Therefore, $R$ approximately satisfies a sub-Gaussian concentration bound; see \cite{wainwright_2019}. By \eqref{eq:var_approx} and \eqref{eq: second_bound} for $t > 0$ this results in
    \begin{equation*}
       \mathrm{\mathbf{Pr}} ( |R - \rho| > t) \lesssim  2 \exp \left \{-\frac{n t^2 }{4 (1 - \rho^2)^2}  \right \} \leq  2 \exp \left (- \frac{n t^2}{4}  \right),
    \end{equation*}
     where substituting the upper bound given by \eqref{eq: var upper bounds} results in \eqref{eq: concervative bound}. We achieve a tighter bound if we accept the approximation 
     $$\sum_{j = m + 1}^{2m} \binom{2m}{j} (-1)^{j} \{E(R^2)\}^{j} (\rho^2)^{m - j} \simeq 0,$$ 
     which results in
    \begin{equation} \label{eq: mega-agressive concentration}
        \mathrm{\mathbf{Pr}} ( |R - \rho| > t) \lesssim 2 \exp \left \{- \frac{n t^2 }{2 (1 - \rho^2)^2 }  \right \} \leq 2 \exp \left (- \frac{n t^2}{2} \right ).
    \end{equation}
    \end{proof}
Interestingly, the bound in \eqref{eq: mega-agressive concentration} links $R$ with a sub-Gaussian random variable with a rate of $n^{-1/2}$, which is usually optimal in parametric problems. Further, note that $\rho$ values close to zero have looser bounds. Hence, $R$ will concentrate more tightly if $|\rho|$ is closer to one (i.e., there is less uncertainty), and if $|\rho| = 1$, then $R = \rho$ almost surely, which the bounds reflect.

\section{Simulation Experiment}
\label{sec: sim experiment}

We simulate ten-thousand $R$ observations coming from samples of size ten of a bivariate normal distribution with population correlation coefficient $\rho$ (i.e., $R_j$ for $j \in \{1, \dots, 10000\}$ is computed for each $j$ from a sample of size ten, where $(X_i, Y_i) \sim \mathcal{N}_2 ( \boldsymbol{\mu}, \mathbf{\Sigma} )$, and $i \in \{1, \dots, 10\}$). Further, we compute the relevant summary statistics and compare them with the approximations given by \eqref{eq: mean approx}, \eqref{eq:var_approx} and the upper bound given by \eqref{eq: var upper bounds}. Subsequently, we compute the percentage of simulated observations that are captured by the coverage intervals given by \eqref{eq: concervative bound}, \eqref{eq: agressive bound} and \eqref{eq: mega-agressive concentration}. The results follow in Tables \ref{tab: table two} thru \ref{tab: cov n = 10} below.

\begin{table}[H]
    \centering
   \begin{tabular}{ccccccc}
$\rho$ & $E(R)$ & $\overline{R}$ & $\mathrm{sd} (R)$ & $s_{R}$ & UB \\
\hline
0 & 0 & -0.003 & 0.333 & 0.332 & 0.471 \\
 -0.25 & -0.237 & -0.236 & 0.312 & 0.316 & 0.450 \\
 0.56 & 0.531 & 0.534 & 0.229 & 0.249 & 0.359 \\
-0.75 & -0.712 & -0.728 & 0.146 & 0.172 & 0.264 \\
 0.95 & 0.901 & 0.944 & 0.033 & 0.045 & 0.11
\end{tabular}
\caption{10,000 $R_j$ simulations computed from samples of size ten (i.e., $n = 10$) coming from $(X_i, Y_i) \sim \mathcal{N}_2 ( \boldsymbol{\mu}, \mathbf{\Sigma} )$ with population correlation coefficient $\rho$. Above $E(R)$ is given by \eqref{eq: mean approx}, $\overline{R}$ and $s_{R}$ are the mean and standard deviation of the simulated values, $\mathrm{sd} (R)$ is given by the square root of \eqref{eq:var_approx}, and UB (upper-bound) is given by the square root of \eqref{eq: var upper bounds}.
}
\label{tab: table two}
\end{table}
We also study the concentration properties of the bounds given by Proposition \ref{prop: concentration inequality}. We solve for $t > 0$ with $n = 10$, and obtain the coverage intervals:
\begin{align} \label{eq: cov intervals}
    C_0 &= \left \{ (-t + \rho, \rho + t) : \enspace 2 \exp \left [- n t^2 \{8 (1 - \rho^2)^2 \}^{-1}  \right ] = 0.05 \right \}, \nonumber \\
    C_1 &= \left \{ (-t + \rho, \rho + t) : \enspace 2 \exp \left [- n t^2 \{4 (1 - \rho^2)^2 \}^{-1} \right ] = 0.05 \right \}, \nonumber \\
    C_2 &= \left \{ (-t + \rho, \rho + t) : \enspace 2 \exp \left [- n t^2 \{2 (1 - \rho^2)^2 \}^{-1} \right ] = 0.05 \right \}.
\end{align}
\begin{table}[H]
    \centering
    \begin{tabular}{cccc}
     $\rho$ & $C_0$ & $C_1$ & $C_2$\\
     \hline
    0 & 100$\%$ & 100$\%$ & 100$\%$ \\
    -0.25 & 100$\%$ & 100$\%$ & 99.3$\%$ \\
    0.56 & 100$\%$ & 99.5$\%$ & 97.3$\%$ \\
    -0.75 & 99.7$\%$ & 98.7$\%$ & 96.1$\%$ \\
    0.95 & 99.1$\%$ & 97.6$\%$ & 95.2$\%$ 
\end{tabular}
    \caption{Estimated coverage computed from 10,000 simulated values of $R$ when $n = 10$, where $C_i$ is given by \eqref{eq: cov intervals}, and $\rho \in \{0, -0.25, 0.56, -0.75, 0.95 \}$. By coverage we mean the percentage of $R_j \in C_i$, where $R_j$ is the $j$th simulation and $i = 0, 1, 2$.
    }
    \label{tab: cov n = 10}
\end{table}

In the interest of replicability, we performed the experiment by fixing the random seed to 2023 in \texttt{R}. The $C_i$ are wider and not contained in $[-1, 1]$ (i.e., $\sup\{C_i \} > 1$ or $\inf\{C_i \} < -1$) if $|\rho|$ is close to zero and/or the sample size is small. Hence, it is difficult to identify uncorrelated random variables with small samples. The above suggests that there is too much uncertainty. However, the coverage intervals become narrower and more precise as sample size increases.  See the supplementary material for results when $n = 3, 5, 30, 100$.

\section{Conclusion}
\label{sec: conclusion}
We presented an approximation for the mean and variance of $R$ under the bivariate normal model assumption. Subsequently, we used these approximations for deriving finite sample concentration inequalities. These inequalities enable us to solve for an error margin $\delta > 0$ with high probability, i.e., $\prob (|R - \rho | \leq \delta) \geq 1 - \alpha$, where $\alpha \in (0, 1)$, when estimating $\rho$. The inequalities reflect the limitations of estimating $\rho$ and identifying uncorrelated random variables with small sample sizes. Finally, the simulation study provided further validation of the results presented in previous sections.

\section*{Acknowledgements}
We thank Heather Battey, Mario Cortina-Borja and Guy Nason for insightful discussions and suggestions in the writing of this work. Salnikov gratefully acknowledges support from the UCL Great Ormond Street Institute of Child Health, Imperial College London, the Great Ormond Street Hospital DRIVE Informatics Programme, the Bank of Mexico and EPSRC NeST Programme grant EP/X002195/1. The views expressed are those of the authors and not necessarily those of the EPSRC.

\addcontentsline{toc}{chapter}{Bibliography}
\bibliographystyle{agsm}
\bibliography{references}

\begin{center}
    {\large\bf Supplementary material}
\end{center}
\appendix
\setcounter{equation}{0}
\counterwithin{equation}{section}
\section{Preliminary approximations}
\label{sec: preliminary apps}
We are interested in approximating the mean of the Pearson sample correlation coefficient, which we assume is computed from a sample of the random variables $(X_i, Y_i) \in \mathbb{R}^2$, for $i \in \{ 1, \dots, n\}$. The sample correlation coefficient $R \in [-1, 1]$ is given by
\begin{equation*}
    R = \frac{\sum_{i = 1}^{n} \big ( X_i - \overline{X} \big ) \big ( Y_i - \overline{Y} \big )}{
    \big \{ \sum_{i = 1}^{n} \big ( X_i - \overline{X} \big ) \big \}^{\frac{1}{2} }
    \big \{ \sum_{i = 1}^{n} \big ( Y_i - \overline{Y} \big ) \big \}^{\frac{1}{2} }} .
\end{equation*}
Furthermore, if we assume that $(X_i, Y_i)$ have a bivariate normal distribution, i.e.,  
    $(X_i, Y_i) \sim \mathcal{N}_2 ( \boldsymbol{\mu}, \mathbf{\Sigma} ),$ 
where $ \boldsymbol{\mu} = (\mu_x, \mu_y) \in \mathbb{R}^2$ and $\mathbf{\Sigma} \in \mathbb{R}^{2 \times 2}$ is the covariance matrix such that $\rho = \sigma_{xy} (\sigma_{x} \sigma_{y} )^{-1} \in [-1, 1]$ is the population correlation coefficient, $\sigma_x$, $\sigma_y > 0$ are the population standard deviations and $\sigma_{xy} \in \mathbb{R}$ is the covariance. Then the probability density function of $R$ is given by
\begin{equation} \label{eq: pdf supp}
    f_{R} (r) = \frac{2^{n - 3} (1 - \rho^2)^{\frac{n - 1}{2}} (1 - r^2)^{\frac{n - 4}{2}} }{
        \pi \Gamma(n - 2) 
    } \sum_{k = 0}^{+\infty} \bigg [\gfrac{n - 1 + k} \bigg ]^2 \frac{\big (2 r \rho \big )^k}{k!} , 
\end{equation}
where $-1 \leq r  \leq 1$; see \cite{heather_notes, fisher_corr}. Hence, the $m$th moment is given by
\begin{equation} \label{eq: mth moment supp}
    E(R^m) = \int_{-1}^{1} r^{m} f_R (r) dr .
\end{equation}
We will use the following result.
\begin{proposition} \label{prop: m moment sum supp}
Let $R$ be the sample correlation coefficient computed from a sample of size $n \geq 3$ coming from a bivariate normal distribution with population correlation coefficient $\rho$. Then the $m$th moment of $R$ is given by
\begin{equation} \label{eq:m_moment_sum supp}
    E(R^m) = \frac{2^{n - 3} (1 - \rho^2)^{\frac{n - 1}{2}}}{
        \pi \Gamma(n - 2) 
    } \sum_{k = 0}^{+\infty} \bigg [\gfrac{n - 1 + k} \bigg ]^2 \frac{\big (2 \rho \big )^k}{k!} g_m(k),
\end{equation}
where $$g_m (k) = \frac{ \gfrac{m + k + 1} \gfrac{n - 2} }{
        \gfrac{n + m + k - 1} } \mathbb{I} \{m + k = 2 q \},$$ $q$ is a non-negative integer and $\mathbb{I}$ is the indicator function.
\end{proposition}
\begin{proof}
    See \cite{hotelling, r_moments} for alternative expressions and derivations. Since $|R| \leq 1$ almost surely, all of its moments exist and we can exchange the order of integration and summation. Before computing the moments we note the following integrals. Each moment requires computing 
    \begin{equation*}
    g_{m}(k) := \int_{-1}^{1} r^{m + k} (1 - r^2)^{\frac{n - 4}{2}} dr .
\end{equation*}
    The integral above is symmetric with respect to zero, and when $(m + k)$ is odd the integrand is an odd function in $[-1, 1]$, thus, in that case it is equal to zero. Now, if $(m +k)$ is even, then the integrand is an even function, so 
\begin{align*}
    g_m(k) &= 2 \int_{0}^{1} r^{m + k} (1 - r^2)^{\frac{n - 4}{2}} dr \\
    &= 2 \int_{0}^{1} (r^2)^{\frac{m + k}{2}} (1 - r^2)^{\frac{n - 4}{2}} dr .
\end{align*}
Now, substitute $u = r^2$, so $\dfrac{dr}{du} = 2^{-1} u^{-\frac{1}{2}}$, above and we have that
\begin{align*}
    &= 2 2^{-1} \int_{0}^{1} u^{\frac{m + k + 1}{2} - 1} (1 - u)^{\frac{n - 2}{2} - 1} du \\
    &= \frac{ \gfrac{m + k + 1} \gfrac{n - 2} }{
        \gfrac{n + m + k - 1} 
    } .
\end{align*}
Therefore, we conclude that
\begin{equation} \label{eq:m_integral supp}
    g_m (k) = \frac{ \gfrac{m + k + 1} \gfrac{n - 2} }{
        \gfrac{n + m + k - 1} } \mathbb{I} \{m + k = 2 q \},
\end{equation}
where $q$ is a non-negative integer and $\mathbb{I}$ is the indicator function. Using \eqref{eq: pdf supp} \eqref{eq: mth moment supp}, we have that
\begin{equation*}
    E(R^m) = \frac{2^{n - 3} (1 - \rho^2)^{\frac{n - 1}{2}}}{
        \pi \Gamma(n - 2) 
    } \sum_{k = 0}^{+\infty} \bigg [\gfrac{n - 1 + k} \bigg ]^2 \frac{\big (2 \rho \big )^k}{k!} g_m(k) .
\end{equation*}
\end{proof}
\subsection*{Gamma function identities}
We will use the following results; see \cite{AbraSteg72}.
\begin{equation} \label{eq:gamma_n supp}
    \gfrac{n - 2} = \frac{ 2^{3 - n} \sqrt{\pi} \Gamma(n - 2) }{
        \gfrac{n - 1}
    }, 
\end{equation}
and for a non-negative integer $q$ 
\begin{align} \label{eq:gamma_q_odd supp}
    \frac{ \Gamma(q + \frac{1}{2}) }{(2q + 1)!} 2^{2q} &= \frac{2^{2q}}{2^{2q}} \frac{(2q)!}{(2q + 1)!} \frac{\sqrt{\pi}}{q!} \nonumber \\
    &= \frac{\sqrt{\pi}}{(2q + 1) \Gamma(q + 1)},
\end{align}
also 
\begin{equation} \label{eq:gam_q_even supp}
    \frac{ \Gamma(q + \frac{1}{2}) }{(2q)!} 2^{2q} =   \frac{\sqrt{\pi}}{\Gamma(q + 1)} .
\end{equation}
Finally, define the function
\begin{equation} \label{eq:kappa supp}
    \frac{\Gamma(z)}{\Gamma(z + \frac{1}{2} )} \frac{\Gamma(z)}{\Gamma(z - \frac{1}{2} )} =: \kappa(z),
\end{equation}
note that for $z > 0$, using Stirling's approximation, we have that 
$$\kappa(z) \simeq \{ 1 - ( z + 2^{-1} )^{-1} \}^{\frac{1}{2}} [ \{ 1 - ( 4z^2)^{-1} \}^{-1} ]^{z - \frac{1}{2}},$$
where $a \simeq b$ indicates that $a$ and $b$, where $a$, $b \in \mathbb{R}$, are asymptotically equal. There exists a $n_0 \in \mathbb{N}$ such that if $n \geq n_0$, then $|a / b - 1| < n^{-1}$, i.e., $\lim_{n \to +\infty} a / b = 1$).

\section{First moment}
\label{sec: first moment}
\begin{proposition} \label{prop: first moment supp}
Let $R$ be the sample correlation coefficient computed from a sample of size $n \geq 3$ coming from a bivariate normal distribution with population correlation coefficient $\rho \in (-1, 1)$. Then, we have that
\begin{equation} \label{eq: mean approx supp}
    E(R) = (1 - n^{-1})^{\frac{1}{2}} \rho + O(n^{-1}) .
\end{equation}
\end{proposition}
\begin{proof}
    
Since $f_R$ is a probability density function, if $m = 0$ we must have that the sum in (\ref{eq:m_moment_sum supp}) adds up to one and ($m + k = k)$, so the even terms can be written as $k = 2q$, where $q$ is a non-negative integer. Also, if $k$ is odd, then that term in \eqref{eq:m_moment_sum supp} is equal to zero (i.e., $g_m(k) = 0$). Hence, by using \eqref{eq:gam_q_even supp}, \eqref{eq:gamma_n supp} and \eqref{eq:gamma_q_odd supp} we have that
\begin{align} \label{eq:sum_0simple supp}
    1 &= \frac{2^{n - 3} (1 - \rho^2)^{\frac{n - 1}{2}} }{
        \pi \Gamma(n - 2) 
    } \sum_{q = 0}^{+\infty} \bigg [\gfrac{n - 1 + 2q} \bigg ]^2 \frac{ \gfrac{ 2q + 1} \gfrac{n - 2} }{ \gfrac{n + 2q - 1} }
        \frac{\big (2 \rho \big )^{2q}}{(2q)!} \nonumber \\
        &=  (1 - \rho^2)^{\frac{n - 1}{2}} \sum_{q = 0}^{+\infty} 
        \frac{ \gfrac{n + 2q - 1}}{ \Gamma(q + 1) \gfrac{n - 1} } (\rho^2)^q.
\end{align}
For the first moment we set $m = 1$ in (\ref{eq:m_moment_sum supp}), thus, for non-vanishing terms in the sum $k = 2q + 1$, where $q$ is a non-negative integer. Further, we have that the integral in (\ref{eq:m_integral supp}) is
\begin{align} \label{eq:g_one supp}
    g_1 (2q + 1) &= \frac{ \Gamma(q + 1 + \frac{1}{2}) \gfrac{n - 2} }{ \gfrac{n + 2q + 1} } \nonumber \\
    &= \bigg (q + \frac{1}{2} \bigg ) 2^{3 - n}  \sqrt{\pi} \Gamma(n - 2) \frac{\Gamma(q + \frac{1}{2})  }{
        \gfrac{n + 2q + 1} \gfrac{n - 1}
    } .
\end{align}
Since $m = 1$, when we substitute (\ref{eq:g_one supp}) into (\ref{eq:m_moment_sum supp}), we have that
\begin{equation} \label{eq:one_nosimple supp}
    E(R) = \frac{(1 - \rho^2)^{\frac{n - 1}{2} }}{\sqrt{\pi}} \sum_{q = 0}^{+\infty} \frac{ \gfrac{2q + n } \Gamma(q + \frac{n}{2}) }{ \gfrac{2q + n + 1} \gfrac{n - 1} } \frac{ \Gamma(q + \frac{1}{2} )  }{(2q + 1)!} 2^{2q} 
    2\bigg(q + \frac{1}{2} \bigg) \rho^{2q + 1} .
\end{equation}
Now, substitute \eqref{eq:gamma_q_odd supp}, \eqref{eq:gamma_n supp} into \eqref{eq:one_nosimple supp} and we have that
\begin{align} \label{eq: 1_app supp}
    E(R) &=  \frac{(1 - \rho^2)^{\frac{n - 1}{2}} }{\gfrac{n - 1} } 
    \sum_{q = 0}^{+\infty} \frac{\Gamma(q + \frac{n}{2} ) }{
    \Gamma(q + \frac{n + 1}{2} ) } \frac{\Gamma(q + \frac{n}{2} ) }{ \Gamma(q + 1) } \frac{2q + 1}{2q + 1} \rho^{2q + 1} \nonumber \\
    &= \frac{(1 - \rho^2)^{\frac{n - 1}{2} }}{\gfrac{n - 1} } \sum_{q = 0}^{+\infty} \frac{\Gamma(q + \frac{n}{2} )}{
    \Gamma(q + \frac{n + 1}{2} ) } \frac{\Gamma(q + \frac{n}{2} )}{
    \Gamma(q + \frac{n - 1}{2} ) } \frac{\Gamma(q + \frac{n - 1}{2} ) }{ \Gamma(q + 1) } \rho (\rho^2)^q \nonumber \\
    &= \rho (1 - \rho^2)^{\frac{n - 1}{2} } \sum_{q = 0}^{+\infty} \frac{\gfrac{n + 2q - 1} }{
   \Gamma(q + 1) \gfrac{n - 1} } \frac{\Gamma(q + \frac{n}{2} )}{
    \Gamma(q + \frac{n + 1}{2} ) } \frac{\Gamma(q + \frac{n}{2} )}{
    \Gamma(q + \frac{n - 1}{2} ) } (\rho^2)^q \nonumber \\
    &= \rho (1 - \rho^2)^{\frac{n - 1}{2} } \sum_{q = 0}^{+\infty} \frac{\gfrac{n + 2q - 1} }{
   \Gamma(q + 1) \gfrac{n - 1} } \kappa \left ( \frac{2q + n}{2} \right ) (\rho^2)^q
\end{align}
Note that above almost has the same summands as (\ref{eq:sum_0simple supp}), however, these are multiplied by a ratio of gammas evaluated at $q + n 2^{-1} \pm 2^{-1}$. Hence, we can approximate these ratios using $\kappa(q + \frac{n}{2}) $. Since $\kappa$ is a monotone non-decreasing function (i.e., $0 < \kappa(z - \epsilon) \leq \kappa(z) \leq 1$ for $z > \epsilon > 0$) we can lower bound the sum by substituting $\kappa(q + \frac{n}{2})$ with 
$\kappa(\frac{n}{2})$, and upper bound it by substituting $\kappa(q + \frac{n}{2})$ with $1$. Note that $\kappa(\frac{n}{2}) \geq (1 - n^{-1})^{\frac{1}{2}}$, thus, using \eqref{eq:kappa supp}, \eqref{eq:sum_0simple supp} and (\ref{eq: 1_app supp}) for $\rho > 0$ we have that
\begin{align} \label{eq:one_ineq supp}
    E(R) &\geq (1 - n^{-1})^{\frac{1}{2}} \rho \sum_{q = 0}^{+\infty} \frac{\Gamma(q + \frac{n - 1}{2} ) }{ \gfrac{n - 1}  \Gamma(q + 1) } (\rho^2)^q (1 - \rho^2)^{\frac{n - 1}{2}} \nonumber \\ 
    &= (1 - n^{-1})^{\frac{1}{2}} \rho, \nonumber \\
    E(R) &\leq \rho \sum_{q = 0}^{+\infty} \frac{\Gamma(q + \frac{n - 1}{2} ) }{ \gfrac{n - 1}  \Gamma(q + 1) } (\rho^2)^q (1 - \rho^2)^{\frac{n - 1}{2}} \nonumber \\ 
    &= \rho.
\end{align}
If $\rho < 0$, then the inequalities in (\ref{eq:one_ineq supp}) turn. Therefore, for $|\rho| < 1$ we conclude that
\begin{equation} \label{eq:one_approx supp}
    (1 - n^{-1})^{\frac{1}{2}} |\rho| \leq |E(R)| \leq |\rho| .
\end{equation}
If $|\rho| = 1$, then $R$ is a degenerate random variable equal to $\pm 1$, hence, in that case $R = \pm 1$ almost surely and $E(R) = \pm 1$.
\end{proof}
Since $E(R)$ and $\rho$ have the same sign, we have that 
$$|E(R) -  (1 - n^{-1})^{\frac{1}{2}} \rho | \leq |\rho| \{1 -  (1 - n^{-1})^{\frac{1}{2}} \} \leq n^{-1}.$$ 
Thus, we have that $E(R) =  (1 - n^{-1})^{\frac{1}{2}} \rho + O (n^{-1})$, which matches Hottelling's approximation error rate; see \cite{hotelling}, but does not depend on $\rho^2$ or any other higher moments.

\section{Second moment}
\label{sec:second moment}
\begin{lemma}{[\cite{hotelling}]} \label{prop:var_approx supp}
    Let $R$ be the sample correlation coefficient computed from a sample of size $n \geq 3$ coming from a bivariate normal distribution with population correlation coefficient $\rho \in (-1, 1)$. Then we have that 
    \begin{equation} \label{eq:var_approx supp}
     \mathrm{var} (R) = (1 - \rho^2)^2 (n - 1)^{-1} + O(n^{-2}).
\end{equation}
\end{lemma}
\begin{proof}
See \cite{hotelling}. 
\end{proof}
Note that the variance is largest when $\rho = 0$ and that if $X$ and $Y$ are perfectly correlated (i.e., $\rho =\pm 1$), then the variance is equal to zero given that $R$ is a degenerate random variable in that case.

\section{Concentration inequalities}
We present a proof of Proposition \ref{prop: concentration inequality}. 
\begin{proof}
    Since $R$ is bounded almost surely, all of its moments exist, moreover, $E \{ (R - \rho)^{2m} \} = E \{ (R - \rho)^{2m - 2} (R - \rho)^{2}\} \leq 2^{m -2} \nu^2$, where $m \in \{1, 2, 3, \dots\}$, and $\nu^2 := \mathrm{var} (R) \leq (n - 1)^{-1}$, where equality holds if $\rho = 0$. Hence, $R$ satisfies Bernstein's condition; see \cite{wainwright_2019}. Consequently, by \eqref{eq:var_approx supp}, for $t > 0$ we have that
    \begin{equation*}
         \prob ( |R - E(R) | > t ) \leq 2 \exp \big [-t^2 \{2 (\nu^2 + 2 t) \}^{-1} \big ],
    \end{equation*}
   scaling properly by $(1 - n^{-1})^\frac{1}{2}$ finishes the first part of the proof.
   \newline
    We proceed to the second part. Combining Markov's inequality, \eqref{eq: mean approx supp} and \eqref{eq:var_approx supp} gives 
    \begin{equation} \label{eq: R consistency supp}
        R = \rho + O_{\prob} (n^{-\frac{1}{2}}).
    \end{equation}
    \par
    Notice that by looking at \eqref{eq:m_moment_sum supp} and expressing the gamma function as a factorial we have that 
    $$g_{2m} (2 q) \leq \prod_{t = 1}^{m - 1} \left ( \frac{2t + 2q + 1}{2t + 2q + n - 1} \right ) g_2 (2q),$$ 
    which when substituted into \eqref{eq:m_moment_sum supp} and noting that 
    $$\left ( 1 - \frac{n - 2}{2m + n - 1} \right ) \simeq \left ( 1 - \frac{n - 2}{2m + n + 2q -3} \right )$$ 
    for $q \in \{1, 2, \dots \}$ results in the inequality
    \begin{equation} \label{eq r2m inequality supp}
       \left ( 1 - \frac{n - 2}{n + 1} \right )^{m -1} E(R^2) \leq E(R^{2m}) \lesssim \left ( 1 - \frac{n - 2}{2m + n - 1} \right )^{m - 1} E(R^2),
    \end{equation}
    thus, we can bound from below and above even moments with the second moment as follows. By Jensen's inequality we have that
    \begin{equation} \label{eq: jensen lhs supp}
         \{E(R^2)\}^m \leq E(R^{2m}).
    \end{equation}
    \par
   Now, by \eqref{eq r2m inequality supp} there exists a sequence $(a_n) \in \mathbb{R}$ such that if $E(R^2) > \left ( 1 - \frac{n - 2}{2m + n - 1} \right )$, then we can find a sequence $a_n \geq 1$ that satisfies
   \begin{equation*}
       E(R^{2m}) \leq a_n  \{E(R^2)\}^m,
   \end{equation*}
    similarly, if $E(R^2) \leq  \left ( 1 - \frac{n - 2}{2m + n - 1} \right )$, then we can find a sequence $a_n \geq 1$ that satisfies 
    $$ \left ( 1 - \frac{n - 2}{2m + n - 1} \right )^{m -1} \leq a_n \{E(R^2) \}^{m -1},$$
    thus, we have that
     \begin{equation*}
        E(R^{2m}) \leq \left ( 1 - \frac{n - 2}{2m + n - 1} \right )^{m - 1} E(R^2) \leq  a_n  \{E(R^2)\}^m,
   \end{equation*}
   so by \eqref{eq: jensen lhs supp} and the above there exists a sequence $(a_n) \in \mathbb{R}$ such that $a_n \geq 1$ and
\begin{equation} \label{eq: close jensen supp}
    \{E(R^2)\}^m \leq E(R^{2m}) \leq  a_n  \{E(R^2)\}^m.
\end{equation}
   
By \eqref{eq: R consistency supp} and the continuous mapping theorem we have that all three terms in \eqref{eq: close jensen supp} tend to $\rho^{2m}$ as $n \to \infty$ if $a_n \to 1$. Hence, there exists a sequence $a_n \geq 1$ that approaches one as $n \to +\infty$ that allows us to perform the approximation:
$$E(R^{2m}) \approx \{E(R^2)\}^m,$$
where $m \in \{1, \dots \}$, and we find $n_0 \in \mathbb{N}$ such that $a_{n_0} \simeq 1$, which holds for any $n \geq n_0$.
\newpage
Using this, and recalling that $|\rho| < 1$ we bound the central even moments as follows.
    \begin{align} \label{eq: 2m moment bound supp}
        E \{ (R - \rho)^{2m} \} &= \sum_{j = 0}^{2m} \binom{2m}{j}  (-1)^{j} E(R^j) \rho^{2m - j} \nonumber \\
        &=  \sum_{j = 0}^{m} \binom{2m}{2j} E(R^{2j}) (\rho^2)^{m - j} \cdots  \nonumber \\
        &- \sum_{j = 0}^{m -1} \binom{2m}{2j + 1} E(R^{2j + 1}) \rho^{2m - 2j - 1} \nonumber \\
        &\overset{(i)}{\lesssim}  \sum_{j = 0}^{m} \binom{2m}{2j} \{E(R^{2}) \}^j (\rho^2)^{m - j} \cdots \nonumber \\
        &- \sum_{j = 0}^{m -1} \binom{2m}{2j + 1} \{E(R^{2})\}^{2j + 1} (\rho^2)^{m - j} |\rho^{-1}| \nonumber \\
        &\overset{(ii)}{\leq} \sum_{j = 0}^{m} \binom{2m}{j} (-1)^{j} \{E(R^2)\}^{j} (\rho^2)^{m - j} \cdots \nonumber \\
        &+ \sum_{j = m + 1}^{2m} \binom{2m}{j} (-1)^{j} \{E(R^2)\}^{j} (\rho^2)^{m - j} \nonumber \\
        &= \sum_{j = 0}^{m} \binom{2m}{j} \binom{m}{j}^{1 - 1} (-1)^{j}\{E(R^2)\}^{j} (\rho^2)^{m - j} \cdots  \nonumber \\
        &+ \sum_{j = m + 1}^{2m} \binom{2m}{j} (-1)^{j} \{E(R^2)\}^{j} (\rho^2)^{m - j},
    \end{align}
    where $(i)$ follows from the approximation given by \eqref{eq: close jensen supp} (i.e., $E(R^{2m}) \approx \{E(R^2)\}^m$), noting that the sign does not alternate for odd summands in subtracting terms, i.e., $\rho^{2m - 2j - 1}$ has the same sign as $E(R^{2j + 1})$, and $(ii)$ follows from rearranging terms and noting that $\rho^{2(m - j)} |\rho^{-1}| > \rho^{2(m - j)}$. Hence, the substitutions above increase the sum by decreasing the subtracting terms in absolute value. Now, we compute the following.
    \begin{align} \label{eq: m sum bound inequality one supp}
        &\sum_{j = 0}^{m} \binom{2m}{j} \binom{m}{j}^{1 - 1} (-1)^{j}\{E(R^2)\}^{j} (\rho^2)^{m - j} \cdots \nonumber \\
        &= \sum_{j = 0}^{m} \frac{(2m)!}{j! (2m - j)!} \cdot \frac{j! (m - j)!}{m!} \binom{m}{j} (-1)^{j} \{E(R^2)\}^{j} (\rho^2)^{m - j} \nonumber \\
        &= \frac{(2m)!}{m!} \sum_{j = 0}^{m} \bigg \{ \prod_{k = 0}^{m} (2m - j - k) \bigg \}^{-1} \binom{m}{j} (-1)^{j} \{E(R^2)\}^{j} (\rho^2)^{m - j} \nonumber \\
        &\leq \frac{(2m)!}{2^{m} m!} \sum_{j = 0}^{m} \binom{m}{j} (-1)^{j} \{E(R^2)\}^{j} (\rho^2)^{m - j} \nonumber \\
        &= \frac{(2m)!}{2^{m} m!} \{ E(R^2) - \rho^2 \}^m \nonumber \\
        &= \frac{(2m)!}{2^{m} m!} \nu^{2m} + O(n^{-1}),
    \end{align}
    where $\nu^2  = \mathrm{var} (R)$ and the last equality follows from substituting \eqref{eq: mean approx supp} into the above. Also, we have that
    \begin{align} \label{eq: m sum boun inequality two supp}
        \sum_{j = m + 1}^{2m} \binom{2m}{j} (-1)^{j} \{E(R^2)\}^{j} (\rho^2)^{m - j} &\leq \sum_{j = 0}^{m} \binom{2m}{j} (-1)^{j} \{E(R^2)\}^{j} (\rho^2)^{m - j} \nonumber \\
        &\leq \frac{(2m)!}{2^{m} m!} \{ E(R^2) - \rho^2 \}^m,
    \end{align}
    hence, using \eqref{eq: 2m moment bound supp}, \eqref{eq: m sum bound inequality one supp} and \eqref{eq: m sum boun inequality two supp}, we have that for $m \in \{1, \dots \}$
    \begin{equation*}
        E \{ (R - \rho)^{2m} \} \lesssim \frac{(2m)!}{2^{m} m!} \left ( 2^{1/2m} \nu \right )^{2m} \leq  \frac{(2m)!}{2^{m} m!} \left (\sqrt{2} \nu \right)^{2m},
    \end{equation*}
    thus, $R$ satisfies a sub-Gaussian type bound; see \cite{wainwright_2019}. Therefore, it approximately satisfies a Gaussian concentration bound, which combined with Lemma \ref{prop:var_approx supp} and for $t > 0$ results in
    \begin{equation*} 
       \mathrm{\mathbf{Pr}} ( |R - \rho| > t) \lesssim  2 \exp \left \{-\frac{t^2 }{4 (1 - \rho^2)^2}  \right \} \leq  2 \exp \left (- \frac{n t^2}{4}  \right).
    \end{equation*}
    \par
    We achieve a tighter bound by approximating the sum in \eqref{eq: m sum boun inequality two supp} with a semi-telescopic sum (i.e., alternating sign terms of very small numbers, albeit multiplied by combinatorial coefficients), which is approximately equal to zero. The approximation is
    $$\sum_{j = m + 1}^{2m} \binom{2m}{j} (-1)^{j} \{E(R^2)\}^{j} (\rho^2)^{m - j} \approx 0,$$ 
    which results in the more aggressive bound for $t > 0$ given by
    \begin{equation} \label{eq: agressive concentration supp}
        \mathrm{\mathbf{Pr}} ( |R - \rho| > t) \lesssim 2 \exp \left (- \frac{t^2 }{2 \nu^{2} }  \right ) \leq 2 \exp \left (- \frac{n t^2}{2} \right ).
    \end{equation}
\end{proof}

\section{Simulation experiment supplement}
\label{sec: simulation experiment}
The results for $n = 3, 5, 30, 100$ follow in Table \ref{tab:corr coeff simulation 2} below. In the interest of replicability, we perform the experiment by fixing the random seed to 2023 in \texttt{R}. Essentially, we simulate 10,000 observations of $R$ computed using samples of size $n = 3, 5, 30, 100$ coming from a bivariate normal distribution with population correlation coefficient $\rho$. In Table \ref{tab:corr coeff simulation 2}: $E(R)$ is given by \eqref{eq:one_approx supp},  $\overline{R}$ is the mean of the simulations, $\mathrm{sd} (R)$ is given by the square root of \eqref{eq:var_approx supp},  $s_{R}$ is the standard deviation of the simulated values. We also study the concentration properties of the bounds given by Proposition \ref{prop: concentration inequality}. We solve for $t > 0$ with $n = 10$, and obtain the coverage intervals:
\begin{align*}
    C_0 &= \left \{ (-t + \rho, \rho + t) : \enspace 2 \exp \left [- \frac{n t^2}{8 (1 - \rho^2)^2 } \right ] = 0.05 \right \}, \\
    C_1 &= \left \{ (-t + \rho, \rho + t) : \enspace 2 \exp \left [- \frac{n t^2}{4 (1 - \rho^2)^2 } \right ] = 0.05 \right \}, \\
    C_2 &= \left \{ (-t + \rho, \rho + t) : \enspace 2 \exp \left [- \frac{n t^2}{2 (1 - \rho^2)^2 } \right ] = 0.05 \right \}.
\end{align*}
Which by Proposition \ref{prop: concentration inequality} give
\begin{equation*}
    \mathrm{\mathbf{Pr}} \left ( -t + \rho \leq R \leq t + \rho \right ) \gtrsim 0.95,
\end{equation*}
where $C_2$ is the tightest bound and $C_0$ is loosest one.
\begin{table}
    \centering
   \begin{tabular}{cccccccc}
   \vspace{4pt}
   &&\multicolumn{5}{r}{10,000 simulations with sample size $n = 3$}\\
$\rho$ & $E(R)$ & $\overline{R}$ & $\mathrm{sd} (R)$ & $s_{R}$ & $C_0$ & $C_1$ & $C_2$ \\ 
\hline
0 & 0 & 0.009 & 0.707 & 0.709 & 1 & 1 & 1 \\
 -0.25 & -0.204 & -0.194 & 0.663 & 0.691 & 1 & 1 & 1 \\
 0.56 & 0.457 & 0.457 & 0.485 & 0.616 & 1 & 1 & 0.92 \\
 -0.75 & -0.612 & -0.646 & 0.309 & 0.507 & 0.977 & 0.93 & 0.892 \\
 \vspace{12pt}
 0.95 & 0.776 & 0.9 & 0.069 & 0.263 & 0.943 & 0.926 & 0.906  \\
    \vspace{3pt}
&&\multicolumn{5}{r}{10,000 simulations with sample size $n = 5$}\\
$\rho$ & $E(R)$ & $\overline{R}$ & $\mathrm{sd} (R)$ & $s_{R}$ & $C_0$ & $C_1$ & $C_2$ \\ 
\hline
0 & 0 & -0.001 & 0.5 & 0.501 & 1 & 1 & 1 \\
 -0.25 & -0.224 & -0.22 & 0.469 & 0.479 & 1 & 1 & 1 \\
 0.56 & 0.501 & 0.512 & 0.343 & 0.4 & 1 & 0.993 & 0.955 \\
 -0.75 & -0.671 & -0.703 & 0.219 & 0.305 & 0.99 & 0.965 & 0.932 \\
 \vspace{12pt}
 0.95 & 0.85 & 0.931 & 0.049 & 0.106 & 0.971 & 0.95 & 0.92  \\
    &&\multicolumn{5}{r}{10,000 simulations with sample size $n = 30$}\\
    \vspace{3pt}
$\rho$ & $E(R)$ & $\overline{R}$ & $\mathrm{sd} (R)$ & $s_{R}$ & $C_0$ & $C_1$ & $C_2$ \\ 
\hline
0 & 0 & 0.003 & 0.186 & 0.184 & 1 & 1 & 0.996 \\
 -0.25 & -0.246 & -0.247 & 0.174 & 0.175 & 1 & 1 & 0.993 \\
 0.56 & 0.551 & 0.555 & 0.127 & 0.129 & 1 & 0.998 & 0.988 \\
 -0.75 & -0.737 & -0.743 & 0.081 & 0.086 & 1 & 0.997 & 0.98 \\
 \vspace{12pt}
0.95 & 0.934 & 0.948 & 0.018 & 0.02 & 0.999 & 0.993 & 0.973  \\
 \vspace{3pt}
     &&\multicolumn{5}{r}{10,000 simulations with sample size $n = 100$}\\
 $\rho$ & $E(R)$ & $\overline{R}$ & $\mathrm{sd} (R)$ & $s_{R}$ & $C_0$ & $C_1$ & $C_2$ \\ 
\hline
0 & 0 & 0.001 & 0.101 & 0.101 & 1 & 1 & 0.993 \\
 -0.25 & -0.249 & -0.25 & 0.094 & 0.095 & 1 & 1 & 0.994 \\
 0.56 & 0.557 & 0.559 & 0.069 & 0.069 & 1 & 0.999 & 0.993 \\
 -0.75 & -0.746 & -0.748 & 0.044 & 0.045 & 1 & 0.999 & 0.989 \\
 0.95 & 0.945 & 0.95 & 0.01 & 0.01 & 1 & 0.999 & 0.988 \\
\end{tabular}
    \caption{
        Simulation experiment with ten-thousand samples of $R$ computed from samples of different sizes coming from a bivariate normal distribution. See text for further description.
    }
    \label{tab:corr coeff simulation 2}
\end{table}

\end{document}